\documentclass{amsart}

\usepackage{amsmath,amssymb,amsthm,booktabs,multirow}
\usepackage[pdfpagelayout=OneColumn, colorlinks=true, urlcolor=black, citecolor=black, linkcolor=black]{hyperref}
\usepackage[all]{xy}

\newtheorem{theorem}{Theorem}[section]
\newtheorem{lemma}[theorem]{Lemma}
\newtheorem{corollary}[theorem]{Corollary}
\newtheorem{proposition}[theorem]{Proposition}
\newtheorem*{thm.strat}{Theorem \ref{thm.strat}}
\newtheorem*{thm.xknonempty}{Theorem \ref{thm.xknonempty}}
\newtheorem*{thm.w1chibounds}{Theorem \ref{thm.w1chibounds}}
\newtheorem*{thm.shext}{Theorem \ref{thm.shext}}

\theoremstyle{definition}
\newtheorem{definition}[theorem]{Definition}
\newtheorem{example}[theorem]{Example}
\newtheorem{remark}[theorem]{Remark}
\newtheorem{problem}[theorem]{Open Question}

\numberwithin{equation}{section}

\newcommand{\ce}{\mathrel{\mathop:}=}                
\newcommand{\ec}{=\mathrel{\mathop:}}                
\newcommand{\op}{\mathcal{O}_{\mathbb{P}^1}}         
\newcommand{\abs}[1]{\left\lvert#1\right\rvert}      
\DeclareMathOperator{\Ext}{Ext}           
\DeclareMathOperator{\Sing}{Sing}         
\DeclareMathOperator{\Tot}{Tot}           
\DeclareMathOperator{\Pic}{Pic}           
\DeclareMathOperator{\Sym}{Sym}           
 
\DeclareMathOperator{\SEnd}{\mathcal{E}\!\mathit{nd}} 

\title{Sheaves on singular varieties}
\author{Elizabeth Gasparim \and Thomas K\"{o}ppe}
\address{School of Mathematics, The University of Edinburgh\\
James Clerk Maxwell Building,
The King's Buildings,
Mayfield Road\\
Edinburgh, EH9 3JZ,
United Kingdom}
\email{Elizabeth.Gasparim@ed.ac.uk}
\email{t.koeppe@ed.ac.uk}


\begin{document}

\begin{abstract}
We prove existence of reflexive sheaves on singular surfaces and
threefolds with prescribed numerical invariants and study their
moduli.
\end{abstract}

\maketitle

\thanks{}

\section{Motivation}\label{sec.motivation}

Sheaves on singular varieties have become very popular recently
because of their appearance in Physics, String Theory and Mirror
Symmetry. In particular, many open questions about sheaves on singular
varieties have come to light. The corresponding mathematical tools,
however, are waiting to be developed. Our aim in this paper is to
entice singularists to develop some basic techniques needed to
approach such questions.

It is extremely common for a physicist or string theorist to start up
a lecture by giving a partition function for a theory, and now even
algebraic geometers are quite often doing the same. It is not just a
fashion, but the fact is that this is an extremely efficient way to
present results. The general format of such partition functions is of
an infinite sum whose terms contain integrals over moduli spaces. Here
are some examples. We will not need details from these expressions,
just the observation that they all contain integrals over moduli
spaces.

\begin{example}\label{Ne}(\textit{String Theory})
The Nekrasov partition function for $N=2$ supersymmetric $SU(r)$ pure
gauge theory on a complex surface $X$ is given by an expression of the
form
\[ Z_X \ce \Lambda^{(1-r)d \cdot d} \sum_{n \geq 0}\Lambda^{2rn}
   \int_{{\mathfrak M}_{r,d, n}(X)}1 \text{ ,} \]
where ${\mathfrak M}_{r,d,n}(X)$ is the moduli space of framed
torsion-free sheaves or rank $r$, and Chern classes $c_1=d$ and
$c_2=n$. For the case of gauge theories with matter, one writes a
similar expression but with more interesting integrands, see
\cite{GL}.
\end{example}

\begin{example}\label{DT}(\textit{Donaldson--Thomas Theory})
For a Calabi--Yau threefold $X$, the partition function for
Donaldson--Thomas theory is given by:
\[ Z_X \ce \sum_{\beta\in H_2(X, \mathbb Z)} \sum_{n\in \mathbb Z}Q^nv^\beta
   \int_{\left[I_n(X,\beta)\right]^\text{vir}} 1 \text{ ,} \]
where $I_n(X,\beta)$ is the moduli space of ideal sheaves $\mathcal I$
fitting into an exact sequence
\[ 0 \longrightarrow {\mathcal I} \longrightarrow {\mathcal O}_X
   \longrightarrow {\mathcal O}_Y \longrightarrow 0 \]
and satisfying 
\[ \chi(\mathcal O_Y)=n \]
and $[Y] = \beta\in H_2(X,\mathbb Z)$, where $\chi$ is the holomorphic
Euler characteristic, see \cite{MNOP}.
\end{example}

\begin{example}\label{GW}(\textit{Gromov--Witten Theory})
For a Calabi--Yau threefold $X$, the partition function for
degree-$\beta$ Gromov--Witten invariants is given by
\[ Z_X \ce \exp \sum_{\beta\neq 0}\sum_{g\geq 0} u^{2g-2}v^\beta
   \int_{\left[\overline{M_g}(X,\beta)\right]^\text{vir}}1 \text{ ,} \]
where $\overline{M_g}(X,\beta)$ is the moduli space of genus-$g$ curves
representing the class $\beta \in H_2(X,\mathbb Z)$. There is a
precise sense in which this partition function is equivalent to the
one in Example~\ref{DT}, see \cite{MNOP}.
\end{example}

These examples illustrate the appearance of integrals over moduli
spaces of sheaves. Even in the case of moduli spaces of maps of
Example~\ref{GW} the theory is still related to a theory given by
integration over moduli of sheaves. Observe that the definition of
moduli spaces itself requires a choice of numerical invariants: in
Example~\ref{Ne} the Chern classes and in Example~\ref{DT} the
Euler characteristic. So, we now agree that we are interested in
moduli spaces of sheaves on surfaces and threefolds. Of course, the
physics motivation is just a bonus, and we could have been interested
in such moduli spaces for purely geometric reasons, as they are part of
classical algebraic geometry. Now physics dictates that we should
consider theories defined over singular varieties. In fact, some of
the most popular categories considered currently by physicists and
string theorists turn out empty in the absence of singularities; such
is the case of the Fukaya--Seidel category and the Orlov category of
singularities. Thus we arrive at the conclusion that we need to
understand moduli of sheaves on singular varieties. Both the case of
global moduli of sheaves on projective varieties and the case of local
moduli on a small neighborhood of a singularity are of interest. For
the local case there is an added difficulty: What are the correct
numerical invariants to be considered? In this paper we will show that
the local holomorphic Euler characteristic provides a satisfactory
invariant for sheaves on local surfaces. For the case of local
threefolds however, the study of numerical invariants is work in
progress, and much remains to be done. The goal of this paper is to
describe partial progress in the understanding of these questions. We
define new numerical invariants for the threefold case, and give
existence of sheaves with given local numerical data.

\section*{Acknowledgements}

We are very happy to contribute to this volume in honour of Andrew
DuPlessis, and thankful to Christophe Eyral for giving us this
opportunity. The final version of this article was completed while the first 
author visited UNICAMP under FAPESP project number 2009/08587-5; 
their hospitality and generous support
is hereby gladly acknowledged.

\section{Main Results}\label{sec.mainres}

In this paper we consider rational surface singularities obtained by
contracting a line $\ell \cong \mathbb{P}^1$ inside a smooth surface
or threefold. Numerous approaches using numerical invariants or
characteristic classes of some sort have been proposed in the past,
see e.g.\ \cite{bra}, and in Section~\ref{inv} we define numerical
invariants, some of which are new, that we need for the present
situation. To set the stage, in Section~\ref{sec.surfaces} we recall
some of our earlier results for sheaves on singular surfaces. The
results for threefolds presented in Section~\ref{3folds} are new and
will appear in more detail in \cite{Ko}.

In Section~\ref{inv} we define the local holomorphic Euler
characteristic $\chi(\ell, \mathcal{F})$ of a reflexive sheaf
$\mathcal{F}$. We will present the following results.

\begin{thm.xknonempty}
Let ${\mathfrak M}_n(X_k)$ be the moduli of reflexive sheaves on
$\mathbb{C}^2\bigl/\mathbb{Z}_k$ with local holomorphic Euler
characteristic equal to $n$. Then for all $n \geq 0$, ${\mathfrak
  M}_n(X_k)$ is non-empty.
\end{thm.xknonempty}


\begin{thm.w1chibounds}
For every rank-$2$ bundle $E$ on $W_1 \ce \Tot\bigl(\op(-1) \oplus
\op(-1)\bigr)$ with $c_1(E)=0$ and $E\rvert_{\mathbb{P}^1} \cong \op(-j)
\oplus \op(j)$, the following bounds are sharp:
\[ 
   j-1
   \leq \chi(\ell, \pi_*E) = \mathbf{h}(E) \leq (j^2+j)(j-1)/6 \text{ .} \]
Here $\pi \colon W_1 \to X$ is the contraction of the zero section $\ell$
and $X$ is the singular threefold $xy - zw = 0$ in $\mathbb C^4$.
\end{thm.w1chibounds}

\begin{thm.shext}
Let $X$ be the singular threefold $xy - zw = 0$ in $\mathbb C^4$. For
each $j \geq 2$ there exists a $(4j-5)$-dimensional family of rank-$2$
reflexive sheaves on $X$ with local holomorphic Euler characteristic $j-1$.
\end{thm.shext}

For each of the cases $j=0$ or $1$, our methods produce 
only the direct images of the split sheaves;
both have local holomorphic Euler characteristic $0$.

\section{Numerical Invariants}\label{inv}

In this section we define numerical invariants for sheaves on a
neighborhood of a singularity. Our first invariant is defined for any
dimension, and is particularly adapted to study reflexive sheaves that
are direct images of vector bundles on a resolution. Let $\pi \colon
(Z, \ell) \to (X,x)$ be a resolution of an isolated quotient
singularity, $\mathcal{F}$ a reflexive sheaf on $Z$ and $n \ce\dim
X$. The following definition is due to Blache, \cite[Def.\ 3.9]{bla1}.

\begin{definition}
The \emph{local holomorphic Euler characteristic} of $\pi_*\mathcal{F}$ at $x$ is
\begin{equation}\label{eq.euler}
  \chi\bigl(x, \pi_*\mathcal{F}\bigr) \ce \chi\bigl(\ell, \mathcal{F}\bigr)
  \ce h^0\bigl(X; \; (\pi_*\mathcal{F})^{\vee\vee} \bigl/ \pi_* \mathcal{F} \bigr)
  + \sum_{i=1}^{n-1}(-1)^{i-1} h^0\bigl(X; \; R^i \pi_* \mathcal{F}\bigr) \text{ .}
\end{equation}
\end{definition}

For the case when $X$ is a compact orbifold, Blache \cite{bla1} shows that
the global Euler characteristics of $X$ and its resolution are related by
\begin{equation}\label{blache}
  \chi\bigl(X, (\pi_*\mathcal{F})^{\vee\vee}\bigr) = \chi\bigl(Z, \mathcal{F}\bigr) +
  \sum_{x \in \Sing X} \chi\bigl(x, \pi_* \mathcal{F}\bigr) \text{ .}
\end{equation}

\begin{example}(\textit{Homological dimension 1})
Consider the case when $Z$ is itself the total space of a bundle on
$\mathbb{P}^1$. Then $Z$ has homological dimension one, and the
expression on the right-hand side of \eqref{eq.euler} reduces to two
terms, which we call the \emph{width} and \emph{height} of $\mathcal{F}$,
respectively:
\begin{equation}\label{eq.hw}
  \mathbf{w}(\mathcal{F}) \ce h^0\bigl(X; \; (\pi_*\mathcal{F})^{\vee\vee} \bigl/ \pi_* \mathcal{F}\bigr)
  \qquad\text{and}\qquad \mathbf{h}(\mathcal{F}) \ce h^0\bigl(X; \; R^1 \pi_* \mathcal{F}\bigr) \text{ .}
\end{equation}
Hence $\chi  = \mathbf{w}+ \mathbf {h}$.
\end{example}

The case when $Z$ is the total space of a negative line bundle on
$\mathbb{P}^1$ was studied in \cite{bgk1} and \cite{gkm}.
Unfortunately, the width vanishes in higher dimensions.

\begin{lemma}\label{width}
\cite[Lemma 5.2]{bgk1} Let $C$ be a curve of codimension $n \geq 2$ in
$Z$ and $\pi \colon Z \to X$ the contraction of $C$ to a point. Then
for any reflexive sheaf $\mathcal{F}$ on $Z$ we have
\[ h^0\bigl(X ; \; (\pi_*\mathcal{F})^{\vee\vee} \bigl/ \pi_* {\mathcal{F}}\bigr) = 0 \text{ .} \]
\end{lemma}

\begin{example}(\textit{Flop})
When $W_1 = \Tot\bigl(\op(-1) \oplus \op(-1)\bigr)$, Lemma~\ref{width}
shows that $\mathbf{w} = 0$. The height is still a non-trivial
invariant, but less powerful than on surfaces.

However, we can define new invariants by restricting to
sub-surfaces. We have two divisors $D_0 \ce \Tot\bigl(\op(-1) \oplus
\{0\} )\bigr)$ and $D_\infty \ce \Tot\bigl(\{0\} \oplus
\op(-1)\bigr)$, which are both isomorphic to $Z_1$, and they span the linear system
\[ \abs{D} \ce \bigl\{ \lambda_0 D_0 + \lambda_\infty D_\infty : [\lambda_0:\lambda_\infty] \in \mathbb{P}^1 \bigr\} \text{ .} \]
Then each $D_\lambda \in \abs{D}$ is isomorphic to $Z_1$, and by restriction to $D_\lambda$
we can define an entire family of pairs $(\mathbf{w}, \mathbf{h})$.
\end{example}

We now return to the case when $Z$ is the total space of a vector
bundle over $\ell = \mathbb P^1$ and there is a contraction $\pi
\colon Z \to X$. We will construct sheaves on $X$ as direct images of
bundles on $Z$, which we now describe. For simplicity, we consider
rank-$2$ bundles with vanishing $c_1$. The general case is no more
difficult, but more unwieldy to present. When $E\rvert_\ell \cong
\op(-j) \oplus \op(j)$, we call the integer $j \geq 0$ the
\emph{splitting type} of $E$. It turns out that the ampleness of the
conormal bundle of $\ell$ implies that $E$ is an algebraic extension
of line bundles,
\begin{equation}\label{eq.extE}
  0 \longrightarrow \mathcal{O}(-j) \longrightarrow E \longrightarrow
  \mathcal{O}(j) \longrightarrow 0 \text{ .}
\end{equation}
A line bundle $\mathcal{O}(n)$ is uniquely determined as the pullback
of $\op(n)$ from $\mathbb{P}^1$, since $\Pic Z \cong \Pic
\mathbb{P}^1$. For every $j\geq0$, there is the trivial extension
$\mathcal{O}(-j) \oplus \mathcal{O}(j)$, which we call the \emph{split
bundle} of splitting type $j$. For convenience, we sometimes write
$E_\text{split}$ for the split bundle of the same splitting type as a
given bundle $E$.

The first cohomology of $\SEnd E$ is finite-dimensional and furnishes
us with our next invariant:
\[ h^1\bigl(Z; \; \SEnd E \bigr) \]
Naturally, we wish to consider the zeroth cohomology as well. Sadly,
this is infinite-dimensional, so extra effort is required. We consider
the $m^\text{th}$ infinitesimal neighbourhood of $\ell$, denoted
$\ell^{(m)}$, which is a projective scheme. The restriction $E^{(m)}
\ce E\rvert_{\ell^{(m)}}$ is coherent. For $i = 0, 1$, we set
\[ \psi_m^i(E) \ce h^i\bigl(\ell^{(m)}; \; E^{(m)} \bigr) \text{ ,} \]
thus $\psi_m^i$ takes finite values. We find that the difference
$\psi_m^i(E_\text{split}) - \psi_m^i(E)$ is eventually constant.

\begin{definition}\label{eq.defD}
For $i = 0, 1$ and $m \gg 0$, set
\[ \Delta_i(E) \ce \psi_m^i(E_\text{split}) - \psi_m^i(E) \text{ .} \]
\end{definition}

For $h^1$, of course, this step is needlessly complicated, as the first
cohomology is actually finite-dimensional, but this way the method may
be applied to spaces in which the conormal bundle of $\ell$ is not
ample.

The two numbers $\Delta_0$ and $\Delta_1$ are related via the Hilbert
polynomial. Recall that for any coherent sheaf $\mathcal{A}$ on a
projective scheme $S$, the Hilbert series
\[ \phi(\mathcal{A},n) \ce \chi\bigl(\mathcal{A}(n)\bigr) \ce \sum_{i \geq 0}
   (-1)^i h^i\bigl(S; \; \mathcal{A}(n)\bigr) \]
is a polynomial of degree $\dim S$. We have
\[ \Delta_0(E) - \Delta_1(E) =  \phi\bigl(E^{(m)}, 0\bigr)
   - \phi\bigl(E^{(m)}_\text{split}, 0\bigr) \text{ .} \]
But the Hilbert polynomials of $E^{(m)}$ and $E^{(m)}_\text{split}$
are the same, as we will show momentarily, and so we have $\Delta_0 =
\Delta_1$, and for computational ease we just stick with $h^1(\SEnd
E)$. The equality of the Hilbert polynomials, and consequently the
fact that the Hilbert polynomial does not see the extension
\eqref{eq.extE}, is a consequence of the following result.

\begin{lemma}\label{lem.hilb-W1}
Let $E$ be an extension of type \eqref{eq.extE} with splitting type
$j$ on either $Z_k \ce \Tot\bigl(\op(-k)\bigr)$ or $W_1 \ce
\Tot\bigl(\op(-1) \oplus \op(-1)\bigr)$. Then the Hilbert polynomial
of $E\rvert_{\ell^{m}}$,
\begin{multline*}
  \phi\bigl(E^{(m)}, n\bigr) = \chi\bigl(E^{(m)}(n)\bigr) \\ \ce \sum_i (-1)^i
  h^i\bigl(\ell^{(m)}; \; E(n)\rvert_{\ell^{(m)}} \bigr) = 
  \begin{cases} (m + 1) (k m + 2 + 2 n) & \text{on $Z_k$, } \\
  \frac13(m+2)(m+1)(2m + 3n + 3) & \text{on $W_1$,} \end{cases}
\end{multline*}
is independent of the extension class, and independent of the splitting
type $j$. Similarly, the Hilbert polynomial of the endomorphism
bundle $\SEnd E\rvert_{\ell^{(m)}}$ is $2 \phi\bigl(E^{(m)}, n\bigr)$.
\end{lemma}
\begin{proof}
By the additivity of the Hilbert polynomial on short exact sequences,
the Hilbert polynomials in question are determined by the Hilbert
polynomial of the line bundles $\mathcal{O}_{\ell^{(m)}}(p)$ for all
$p$. Since $\mathcal{O}_{\ell^{(m)}}(1)$ is ample, the higher
cohomology of $\mathcal{O}_{\ell^{(m)}}(p)$ vanishes for sufficiently
large $p$. (We can verify this by direct computation.)

Being a polynomial, the Hilbert polynomial is determined by finitely
many values, so it suffices to compute $\phi(\SEnd E^{(m)}, n) =
h^0\bigl(\ell^{(m)};\;\mathcal{O}_{\ell^{(m)}}(p)\bigr)$ for large
$p$. Since $E$ and $\SEnd E$ have filtrations by line bundles, which
restrict to filtrations on every infinitesimal neighbourhood
$\ell^{(m)}$, we compute:
\begin{eqnarray*}
  \phi\bigl(E^{(m)}, n\bigr)       &=& \phi\bigl(\mathcal{O}_{\ell^{(m)}}(-j), n\bigr) + \phi\bigl(\mathcal{O}_{\ell^{(m)}}(j), n\bigr) \text{ , and} \\
  \phi\bigl(\SEnd E^{(m)}, n\bigr) &=& \phi\bigl(\mathcal{O}_{\ell^{(m)}}(-2j), n\bigr) + 2\phi\bigl(\mathcal{O}_{\ell^{(m)}}, n\bigr)
                                       + \phi\bigl(\mathcal{O}_{\ell^{(m)}}(2j), n\bigr) \text{.}
\end{eqnarray*}

We conclude this proof by computing $H^0\bigl(\ell^{(m)}; \;
\mathcal{O}(p)\bigr)$. Now we have to consider the spaces $Z_k$ and
$W_1$ separately. We pick a chart $U$ with local coordinates $(z,u)$
on $Z_k$ and $(z,u,v)$ on $W_1$, respectively, which transform to
$(z^{-1}, z^ku)$ and $(z^{-1},zu,zv)$.

On $\ell^{(m)} \subset Z_k$, a section $a \in \mathcal{O}(p)(U)$ is a
function $a(z,u) = \sum_{r=0}^m \sum_{s=0}^\infty a_{rs} z^s u^r$ such
that $\sum_{r,s} a_{rs} z^{s-p}u^r$ is holomorphic in $(z^{-1},z^ku)$,
i.e.\ $s-p\leq kr$. Thus
\[ a(z,u) = \sum_{r=0}^m \sum_{s=0}^{kr+p} a_{rs} z^s u^r \text{ ,} \]
which has $\frac12(m + 1)(k m + 2 + 2 p) \ec \phi_{\mathcal{O}}(p)$ coefficients.

On $\ell^{(m)} \subset W_1$, a section $a \in \mathcal{O}(p)(U)$ is
$a(z,u,v) = \sum_{t=0}^{m} \sum_{r=0}^{m-t} \sum_{s=0}^\infty a_{trs} z^s u^r v^t$ such that
$\sum_{t,r,s} a_{trs} z^{s-p}u^r v^t$ is holomorphic in $(z^{-1},zu,zv)$,
i.e.\ $s-p\leq r+t$. Thus
\[ a(z,u,v) = \sum_{t=0}^{m\vphantom{p}} \sum_{r=0}^{m-t\vphantom{p}}
   \ \ \sum_{s=0}^{r+t+p} a_{trs} z^s u^r v^t \text{ ,} \]
which has $\frac16(m+2)(m+1)(2m+3p+3) \ec \phi\bigl(\mathcal{O}, p\bigr)$
coefficients.

Putting it all together, we have
\begin{eqnarray*}
  \phi\bigl(E^{(m)}, n\bigr) &=& \phi\bigl(\mathcal{O}, -j+n\bigr) + \phi\bigl(\mathcal{O}, j+n\bigr) \text{ ,} \\
  \phi\bigl(\SEnd E^{(m)}, n\bigr) &=& \phi\bigl(\mathcal{O}, -2j+n\bigr) + 2\phi\bigl(\mathcal{O}, n\bigr)
      + \phi\bigl(\mathcal{O}, 2j+n\bigr) \text{ ,}
\end{eqnarray*}
which gives the desired functions.
\end{proof}

\subsection{Examples of invariants}

To make the notion of the numbers we defined above more concrete, we
tabulate examples for the two bundles $E = \mathcal{O}(-3) \oplus
\mathcal{O}(3)$ (the split bundle of splitting type $3$) and $G$, the
``most generic'' bundle of splitting type $3$ (which has the lowest
invariants among all bundles of splitting type $3$), on the spaces
$Z_1$, $Z_2$, $Z_3$ and $W_1$; see Table~\ref{tab.invexamples}.

\begin{table}
\begin{center}
\begin{tabular}{lp{9pt}cccp{9pt}ccc}\toprule
 && $w(E)$ & $h(E)$ & $h^1(\SEnd E)$ && $w(G)$ & $h(G)$ & $h^1(\SEnd G)$ \\ \midrule
$Z_1$ && $6$ & $3$ & $15$  &&  $1$ & $2$ & $9$ \\
$Z_2$ && $2$ & $2$ & $9$  &&  $0$ & $2$ & $7$ \\
$Z_3$ && $1$ & $2$ & $7$  &&  $0$ & $2$ & $6$ \\
$W_1$ && $0$ & $4$ & $35$ &&  $0$ & $2$ & $17$ \\ \bottomrule
\end{tabular}
\end{center}
\caption{The invariants width, height and $h^1(\SEnd)$ for the split bundle $E$
         and a generic bundle $G$ of splitting type $j = 3$ on the spaces $Z_1$,
         $Z_2$, $Z_3$ and $W_1$.}
\label{tab.invexamples}
\end{table}

\section{Surfaces}\label{sec.surfaces}

Let $Z_k \ce \Tot\bigl(\op(-k)\bigr)$ and let $E$ be a rank-$2$ bundle
on $Z_k$ with $c_1(E) = 0$ and splitting type $j$. Then $E$ is
determined by an element $p \in \Ext^1\bigl(\mathcal O(j), \mathcal O(-j)\bigr)$
as in \eqref{eq.extE}.
The direct image $\pi_*(E)$ is a reflexive sheaf on $X_k$, and there
are bounds for its local holomorphic Euler characteristic around the
singular point $x \in X_k$ in terms of $j$. An efficient algorithm to
compute $\mathbf{w}, \mathbf{h}$ and $\chi$ is given in
\url{http://www.maths.ed.ac.uk/~s0571100/Instanton/}, hence we can
explicitly calculate the values of these numerical invariants for any
such bundle $E$. We present here a useful existence result.

\begin{lemma}\label{jlessk}
Let $E$ be a rank-$2$ bundle over $Z_k$, $k>1$, with $c_1(E) = 0$ and splitting
type $j < k$. Then
\[ \chi(x, \pi_*E) = j - 1 \text{ .} \]
\end{lemma}
\begin{proof}
By \cite[Theorem 3.3]{ga} it follows that if $j<k$ then $E \cong
\mathcal{O}_{Z_k}(j) \oplus \mathcal{O}_{Z_k}(-j)$. By definition,
$\chi(x, \pi_*E) = \mathbf{w}(E) + \mathbf{h}(E)$. Direct computation
(see \cite{bgk1}) then shows that $\mathbf{w}(E) = 0$ and
$\mathbf{h}(E) = j - 1$.
\end{proof}

In fact, we can say a lot more.

\begin{lemma}\label{lem.chibounds}\cite[Corollary 2.18]{bgk1}
Let $E$ be a rank-$2$ bundle over $Z_k$, $k>1$, with splitting type $j
> 0$. Set $j = qk + r$ with $0 \leq r <q$. The following bounds are sharp:
\[ j - 1 \leq \chi(x, \pi_*E) \leq \begin{cases} q^2k+(2q+1)r-1 &
   \text{if }   1 \leq r < k \text{ ,} \\
   q^2k & \text{if }  r = 0 \text{ .} \end{cases} \]
\end{lemma}

\begin{remark}
Note that every bundle that satisfies the conditions of Lemma~\ref{jlessk}
is split, whereas in general there are many distinct isomorphism
classes of bundles, which attain a whole range of numerical
invariants. The lower bound in Lemma~\ref{lem.chibounds} is attained
by a class of generic bundles, while the upper bound is obtained by
the split bundle of splitting type $j$, and moreover, the split bundle
is the only bundle to attain the bound when $r = 0$.
\end{remark}

These two lemmas directly imply the following existence result.

\begin{theorem}\label{thm.xknonempty}
Let ${\mathfrak M}_n(X_k)$ be the moduli of reflexive sheaves on $X_k$
with local holomorphic Euler characteristic equal to $n$. Then for all
$n \geq 0$, ${\mathfrak M}_n(X_k)$ is non-empty.
\end{theorem}

\medskip

\subsection{Applications to physics}

To illustrate applications to physics, we mention some results on the
existence of instantons. We stress that this particular instance of
gaps on instantons charges presented below was completely new to
physicists. In fact, there was a folklore belief that $1$-instantons
are always the most common, and that higher instantons of charge $k$
should decay to $k$ instantons of charge $1$ over time. Our results
showed that over the spaces $Z_k$ with $k \geq 3$ there do not exist
any $1$-instantons, nevertheless higher charge instantons do exist (of
course we mean mathematical existence proofs).

In \cite[Proposition 54]{gkm} we studied the Kobayashi--Hitchin
correspondence for the spaces $Z_k$: We showed that an
$SU(2)$-instanton on $Z_k$ of charge $n$ corresponds to a holomorphic
$SL(2)$-bundle $E$ on $Z_k$ with $\chi(\ell, E) = n$ together with a
trivialization of $E\vert_{Z_k^\circ}$, where $Z_k^\circ \ce Z_k -
\ell$. A simple observation \cite[Proposition 4.1]{gkm} shows that
there exists a trivialization of $E\rvert_{Z_k^\circ}$ if and only if
$n = 0 \mod k$. This restricts the splitting type of an instanton
bundle over $Z_k$ to be of the form $nk$ and lead us to the following
existence/non-existence result:

\begin{proposition}\label{mini}\cite[Theorem 6.8]{gkm}
The minimal local charge of a non-trivial $SU(2)$-instanton on $Z_k$ is
$\chi^\mathrm{min}_k = k - 1$. The local moduli space of (unframed)
instantons on $Z_k$ with fixed local charge $\chi^\mathrm{min}_k$
has dimension $k - 2$.
\end{proposition}

This result shows a straightforward passage from the algebraic
geometry of bundles on surfaces to meaningful mathematical physics.
Similar results for Calabi--Yau threefolds promise to have exciting
interpretations in string theory and physics, whenever the
mathematical background is constructed.

\begin{remark}(\textit{Gaps of instanton charges})
The non-existence of instantons with certain local charges on the
spaces $Z_k$ for $k>2$ is in stark contrast with what happens in the
case $k = 1$, where there is no gap \cite[Theorem 0.2]{BG1}.
\end{remark}

\begin{problem}
Theorem \ref{lem.chibounds} gives sharp bounds for $\chi$ -- are the
intermediate values achieved? Given an integer $\alpha$ such that
$j-1 < \alpha < q^2k$, does there exist an instanton bundle on $Z_k$
with splitting type $j$ and $\chi = \alpha$?  We have a positive
answer for analogous question when $k = 1$, all other cases are open.
\end{problem}

We illustrate also an application to topology:

\begin{theorem}\label{thm.strat}\cite[Theorem 4.15]{bgk1}
If $j=qk$ for some $q \in \mathbb{N}$, then the pair $(\mathbf{w}, \mathbf{h})$
stratifies instanton moduli stacks $\mathfrak{M}_{j,k}$ into Hausdorff
components.
\end{theorem}

\begin{problem}
Find invariants that stratify the moduli stacks $\mathfrak{M}_{j,k}$ in
the case $j=nk+r$ with $r \neq 0 \mod{k}$. We know that the pair
$(\mathbf{w}, \mathbf{h})$ does not provide a fine enough invariant to
stratify the moduli stacks in these cases. Thus, some extra numerical
invariant is needed. At the moment the authors are completely unaware
of any suitable candidate.
\end{problem}

We find it completely surprising that the case $r=0$, whose   
physics interpretation is known, turned out to be much simpler to 
solve. From a topological point of view one should of course 
have Hausdorff stratifications for the moduli stacks in all cases.

\section{Threefolds}\label{3folds}

Consider a smooth threefold $W$ containing a line $\ell \cong \mathbb{P}^1$.
We will focus on  the Calabi-Yau cases 
\[ W_i \ce \Tot(\op(-i) \oplus \op(i-2) \text{ for $i=1,2,3$.} \]
The existence of a contraction of $\ell$ imposes heavy restrictions on
the normal bundle \cite{ji}, namely $N_{\ell/W}$ must be isomorphic to
one of
\[ \text{(a) \ } \op(-1) \oplus \op(-1)
   \text{ ,}\qquad \text{(b) \ } \op(-2) \oplus \op(0)
   \text{ , \ or} \qquad \text{(c) \ } \op(-3) \oplus \op(+1)
   \text{ .} \]
Conversely, Jim\'enez states that if $\mathbb{P}^1 \cong \ell \subset W$
is any subspace of a smooth threefold $W$ such that $N_{\ell/W}$ is
isomorphic to one of the above, then:
\begin{itemize}
\item in (a) $\ell$ always contracts, 
\item in (b) either $\ell$ contracts or it moves, and 
\item in case (c) there exists an example in which $\ell$ does not contract
      nor does any multiple (i.e.\ any scheme supported on $\ell$) move.
\end{itemize}

$W_1$ is the space appearing in the basic flop. Let $X$ be the cone
over the ordinary double point defined by the equation $xy - zw = 0$
on $\mathbb{C}^4$. The basic flop is described by the diagram:
\begin{equation}\label{eq.flop}
 \xygraph{  !{<0cm,0cm>;<1cm,0cm>;<0cm,1cm>::}
   !{(0,1.2)}*+{W}="t"
   !{(1,0)}*+{W_1^+}="l"
   !{(-1,0)}*+{W_1^-}="r"
   !{(0,-1.2)}*+{X}="b"
   !{(0,0)}="m"
   "t":"l"^{p_2} "t":"r"_{p_1} 
   "r":"b"_{\pi_1} "l":"b"^{\pi_2} "l":@{<--}"r"|\hole "t":"b"
   } 
\end{equation}
Here $ {W} \ce {W}_{x,y,z,w}$ is the blow-up of $X$ at the vertex
$x=y=z=w=0$, $W_1^- \ce Z_{x,z}$ is the small blow-up of $X$ along
$x=z=0$ and $W_1^+ \ce Z_{y,w}$ is the small blow-up of $X$ along
$y=w=0$. The \emph{basic flop} is the rational map from $W^-$ to
$W^+$.

In $W_2 \cong Z_2 \times \mathbb{C}$ the zero section does not
contract to a point (so it must be able to move), but it is possible
to contract it partially and obtain a singular family $X_2 \times
\mathbb{C}$, where $X_2$ is the surface containing an ordinary
double-point singularity defined by $xy - z^2 = 0$ in $\mathbb{C}^3$.
Holomorphic bundles on $W_2$ have infinite local holomorphic Euler
characteristic, but the restriction $E\rvert_{Z_2\times\{0\}}$ has
well-defined and finite width and height. Note that in contrast to
$W_1$, there are strictly holomorphic (non-algebraic) bundles on
$W_2$, although every rank-$2$ bundle on $W_2$ is still an extension
of line bundles.

In $W_3$ not even a partial contraction of the zero section is
possible. Nevertheless we can still calculate the width and height of
the restriction $E\rvert_{Z_3}$ of a bundle $E$ to a subsurface $Z_3
\hookrightarrow W_3$. Again, on $W_3$ there are strictly holomorphic
(non-algebraic) bundles, and moreover, there are (many) rank-$2$
bundles which are not extensions of line bundles.

\subsection{Bounds and generating functions}

We can compute the invariants $\mathbf{w}(E)$, $\mathbf{h}(E)$ and
$h^1(\SEnd E)$ directly and algorithmically. We have an implementation
of each of the algorithms for the commutative algebra software
\emph{Macaulay~2}, which led us to discover several formulae for the
bounds of these invariants. Bounds for the local holomorphic Euler
characteristic $\chi = \mathbf{w} + \mathbf{h}$ on surfaces were
presented in Section~\ref{sec.surfaces}; now we turn to the flop space
$W_1$, were by Lemma~\ref{width}, we have $\chi = \mathbf{h}$.

\begin{theorem}\label{thm.w1chibounds}
For every rank-$2$ bundle $E$ on $W_1$ with $c_1(E)=0$ and splitting type $j$,
the following bounds are sharp:
\[
  j-1
  \leq \chi(\ell, E) = \mathbf{h}(E) \leq (j^2+j)(j-1)/6 \text{ .} \]
\end{theorem}
\begin{proof}
The lower bound is attained by a class of generic bundles, and the upper
bound by the split bundle $\mathcal{O}(-j) \oplus \mathcal{O}(j)$. This
can be seen by direct computation as explained in \cite{bgk1} and \cite{Ko}.
\end{proof}

We also have a concise expression for the numbers $h^1(\SEnd)$ of the
extremal cases, that is generic and the split bundles of splitting
type $j$.

\begin{definition}
A power series of the form $g(z) = \sum_{j=0}^\infty a_j z^j$ is
called a \emph{generating function} for the sequence
$(a_j)_{j=0}^\infty$. Hence, $\displaystyle a_j = \frac{1}{j!}
\frac{d^jg}{dz^j} \Bigl\rvert_{z=0}$.
\end{definition}

Set $a^{X,E}_j \ce h^1(X; \; \SEnd E)$. Then if the base space is $X =
Z_k$ or $W_1$ and the bundle $E$ over $X$ is either split or generic
of splitting type $j$, we have generating functions for $a^{X,E}_j$,
as shown in Table~\ref{tab.genfun}. Since the generating function of a
sum of two sequences is the sum of the generating functions, we can
easily deduce from this the generating functions for $\Delta_0$ and
$\Delta_1$. We spell out the inequalities.

\begin{theorem}\label{thm.w1h1bounds}
For every rank-$2$ bundle $E$ on $W_1$ with $c_1(E)=0$ and splitting type $j$,
the following bounds are sharp:
\[ (j^3+3j^2-j)/3 \leq h^1(W_1; \; \SEnd E) \leq (4j^3-j)/3 \]
\end{theorem}
\begin{proof}
The lower bound is attained by a generic bundle and the upper bound by the
split bundle, the values are found by direct computation.
\end{proof}

\begin{table}
\begin{center}
\begin{tabular}{ccc}\toprule
Space & Split bundle $E_j$ & Generic bundle $G_j$ \\\midrule\\[.25ex]
$Z_k$, $k=2n$ & $\displaystyle\frac{-z(z^{n+1}+z^n+z+1)}{(z-1)^2 (z^k-1)}$ &
  \multirow{2}{*}{\parbox{3.25cm}{\bigskip$\displaystyle\frac{z^{k+2}-z^3-z^2-z}{(z-1)^2(z^k-1)}$}} \\[5ex]
$Z_k$, $k=2n+1$ & $\displaystyle\frac{-z(2z^{n+1}+z+1)}{(z-1)^2(z^k-1)}$ & \\[5ex]
$W_1$ & $\displaystyle\frac{z(z^2+6z+1)}{(z-1)^{4}}$ & $\displaystyle\frac{z(-z^2+2z+1)}{(z-1)^{4}}$ \\[.25ex]\\\bottomrule
\end{tabular}\bigskip
\end{center}
\caption{Generating functions for $a_j^{X,E} \ce h^1(X; \; \SEnd E)$
  on various spaces for the split and the generic bundle of splitting
  type $j$ (data for $G_j$ only valid for $j \geq k$); the value
  $a_j^{X,E}$ is the $j^\text{th}$ coefficient in the Taylor series.}
\label{tab.genfun}
\end{table}

\subsection{Moduli of sheaves}

We consider sheaves on singular varieties obtained as direct images of
bundles on $W_i$. First we study such bundles and their moduli. The
topological structure of these moduli is not yet well understood. Most
numerical invariants defined in Section~\ref{inv} can be computed over
any $W_i$; however, the invariants $\Delta_0$ and $\Delta_1$ in
\eqref{eq.defD} are infinite on $W_2$ and $W_3$, so more refined
counterparts are required.

\begin{problem}
Construct a Hausdorff stratification of the moduli stacks
$\mathfrak{M}_n(W_i)$ of bundles on $W_i$ with $c_1 = 0$ and
$\chi(\ell, E) = n$.
\end{problem}

We obtain a partial understanding of these moduli by looking at
first-order deformations, and this will provide enough bundles for an
existence theorem of reflexive sheaves on the corresponding singular
varieties.

\begin{proposition}{\rm(}\textit{First-order deformations}{\,\rm)}
Set $F \ce {\mathcal O}_\ell(-j) \oplus {\mathcal O}_\ell(j)$
with $\ell \subset W_i$.
\begin{enumerate}
\item For any bundle $E$ on $W_i$ with $E\rvert_\ell \cong F$, the space of
      first-order deformations of $G$ is isomorphic to $\mathbb{C}^{\gamma_1}$, where
      \[ \gamma_1 \ce h^1 \bigl(\ell; \; \SEnd(E\rvert_{\ell}) \otimes \mathcal{I}_\ell
         \bigl/ \mathcal{I}_\ell^2 \bigr) < \infty \text{ .} \]
\item If $\mathcal{I}_\ell \bigl/ \mathcal{I}_\ell^2$ is ample (i.e.\ if $i=1$),
      then there exists a vector bundle $A$ on $W_1$ such that $A\rvert_\ell \cong F$.
\end{enumerate}
\end{proposition}
\begin{proof}
The dimension count is standard deformation theory. Existence of
extensions to formal and small analytic neighbourhoods of $\ell$ are
given by Peternell's Existence Theorem \cite{pet2}. The fact that we
actually get existence on the entire space $W_1$ rather than just a
small neighbourhood of $\ell$ is due to the fact that every bundle on
$W_1$ is determined by its restriction to a finite infinitesimal
neighborhood of ${\ell}$.
\end{proof}

\begin{corollary}\label{dim}{\rm(}\textit{Dimension of moduli}{\,\rm)}
The moduli space of first order deformations of ${\mathcal O}(j)
\oplus {\mathcal O}(-j)$ over $W_i$ modulo holomorphic isomorphisms is
isomorphic to $\mathbb{P}^{4j-5}$.
\end{corollary}

\begin{proof}
It is well known that multiplying the extension class by a non-zero
constant does not change the holomorphic type of the underlying
bundle. It turns out that on the first formal neighborhood this is the
only identification. This was proved for surfaces in \cite[Theorem~4.9]{bgk1}
and for $W_i$, $i=1,2,3$ in \cite{Ko}. We can then compute $\gamma_1$
directly as the dimension of the first cohomology of
$\SEnd\bigl(\op(-j) \oplus \op(j)\bigr) \otimes N^*_{\ell/W_i}$ on
$\mathbb{P}^1$. The $\SEnd$-bundle splits into a direct sum of line
bundles, and the computation is straightforward.
\end{proof}

If instead of the first-order deformations we wish to consider all
deformations, then the dimension of the deformation space is given by
\begin{equation}\label{eq.gamma}
  \gamma \ce \sum_{m=0}^\infty h^1 \bigl(\ell; \; \SEnd(E\rvert_{\ell})
  \otimes \Sym^m(\mathcal{I}_\ell \bigl/ \mathcal{I}_\ell^2) \bigr) \text{ ,}
\end{equation}
which is finite when $\mathcal{I}_\ell \bigl/ \mathcal{I}_\ell^2$ is
ample, but infinite in general. Though the space of deformations may
be infinite, it turns out that for a fixed $j$
the moduli space $\mathfrak{M}_n(W_i)$ of holomorphic
bundles $E$ on $W_i$ with $\chi(\ell, E) = n =
j-1
$
has a Zariski-open set of dimension $4j-5$ consisting of of
first-order deformations of $\mathcal{O}(j) \oplus \mathcal{O}(-j)$
(cf.\ Corollary~\ref{dim}).
Now, using these moduli for the case of $W_1$, we obtain sheaves on
the singular threefold $X$ appearing on the flop diagram
\eqref{eq.flop}.

\begin{theorem}\label{thm.shext}
Let $X$ be the singular threefold $xy - zw = 0$ in $\mathbb C^4$. For
each $j \geq 2$ there exists a $(4j-5)$-dimensional family of rank-$2$
reflexive sheaves on $X$ with local holomorphic Euler characteristic $j-1$.
\end{theorem}
\begin{proof}
These reflexive sheaves are obtained as direct images of generic
bundles on $W_1$ with splitting type $(-j,j)$. Combine
Corollary~\ref{dim} with the value of $\chi$ found for the generic
bundle as given in Table~\ref{tab.genfun}.
\end{proof}

For the case of $j=0$ or $1$ our methods give only the direct images
of the split bundles $\mathcal{O} \oplus \mathcal{O}$ and
$\mathcal{O}(1) \oplus \mathcal{O}(-1)$, both have $\chi=0$.

We stop short of stating a similar theorem for the singular spaces
obtained by partial contractions on $W_i$ with $i=2,3$ because
strictly speaking the definition of local Euler characteristic was
given for isolated singularities. We do obtain existence of reflexive
sheaves on those spaces, but we do not yet have a good feel for what
would be the correct numerical numerical invariants to use.

\begin{problem}
Describe the full moduli of reflexive rank-$2$ sheaves on $W_1$ with
$c_1 = 0$ and $\chi = n$, that is, include all sheaves that do not occur
as direct images of bundles on $W_1$.
\end{problem}

\begin{problem}
Describe moduli of sheaves with fixed numerical invariants on germs of
singularities.
\end{problem}

The latter is of course a very big question, actually infinitely many
open questions, starting with the definition of the correct invariants
up to their computation and then construction of moduli. It is certainly an
entire research project for a whole group of singularists. We hope
some singularists get inspired to work on these questions.

\end{document}